\documentclass[11pt]{amsproc}
\usepackage{graphicx}
\usepackage{amssymb}
\usepackage{epstopdf}
\usepackage{amsmath,amsthm,amscd,amssymb}
\usepackage{latexsym}
\usepackage[colorlinks,citecolor=red,urlcolor=blue,pagebackref,hypertexnames=false]{hyperref}
\usepackage{geometry}                
\numberwithin{equation}{section}

\theoremstyle{plain}
\newtheorem{theorem}[equation]{Theorem}
\newtheorem{lemma}[equation]{Lemma}

\theoremstyle{definition}
\newtheorem{definition}[equation]{Definition}

\newtheorem{properties}[equation]{Properties}

\theoremstyle{remark}
\newtheorem{remark}[equation]{Remark}

\newcommand{\I}{\operatorname{I}}
\newcommand{\II}{\operatorname{II}}

\newcommand{\rn}{\mathbb{R}^n}

\begin{document}

\title[A local $Tb$ theorem for square functions in domains with ADR boundaries]{A local $Tb$ theorem for square functions in domains with Ahlfors-David regular boundaries}

\author{Ana Grau de la Herran}
\address{DEPARTMENT OF MATHEMATICS, UNIVERSITY OF MISSOURI, COLUMBIA}
\email{ag432@mail.missouri.edu}
\author{Mihalis Mourgoglou}
\address{UNIV. PARIS-SUD 11, LABORATOIRE DE MATH\'{E}MATIQUES, UMR 8628, F-91405 {\sc ORSAY}; FMJH, F-91405 {\sc ORSAY}}
\email{mihalis.mourgoglou@math.u-psud.fr}

\begin{abstract} We prove a “local” $Tb$ Theorem for square functions, in which we assume $L^p$ control of the pseudo-accretive system, with $p > 1$ extending the work of S. Hofmann to domains with Ahlfors-David regular boundaries. 
\end{abstract}

\maketitle

\tableofcontents

\section{Introduction, statement of results, history\label{s1}}

The $Tb$ Theorems are boundedness criteria for singular integrals which show that a singular integral operator $T$ is $L^2$ bounded if it has sufficiently good behavior on some suitable non-degenerate test function $b$. Such theorems were first proved in $\rn$ by McIntosh and Meyer in \cite{McM}, and David, Journ\'{e} and Semmes in \cite{DJS}. A “local $Tb$ theorem” is a variant of the standard $Tb$ theorem, in which instead of controlling the action of the operator $T$ on a single, globally defined accretive test function $b$, it suffices to control locally, on each dyadic cube $Q$, the action of $T$ on a test function $b_Q$, which satisfies some uniform, scale invariant $L^p$ bound along with the non-degeneracy condition
$$\frac{1}{C} \leq |Q|^{-1} \left|\int_Q b_Q \right|,$$
for some uniform constant $C>0$. M. Christ proved in \cite{Ch} the first local $Tb$ theorem, in which the local test functions are assumed to belong uniformly to $L^{\infty}$; an extension of Christ's result to the non-doubling setting is due to Nazarov, Treil and Volberg \cite{NTV} and Hyt\"{o}nen and Martikainen \cite{HyM}. For doubling measures, one can also consider more general $L^p$ type testing conditions introduced by Auscher, Hofmann, Muscalu, Tao and Thiele \cite{AHMMT},
and further studied by Hofmann \cite{H2}, Auscher and Yang \cite{AY}, Auscher and Routin \cite{AR}, Hyt\"{o}nen and Martikainen \cite{HyM} and Tan and Yan \cite{TY}.

It is also interesting to consider local Tb theorems for square functions (as opposed to singular integrals). In the Euclidean setting the result obtained in our work was presented in \cite{H1}, but was already implicit in the solution of the Kato problem \cite{HMc}, \cite{HLMc}, \cite{AHLMcT}, (see also \cite{AT} and \cite{Se} for related results), and more recently has found application to variable coefficient layer potentials \cite{AAAHK}. Our result is applicable to problems that connect the behavior of the harmonic measure for domains with uniformly rectifiable boundaries (see \cite{HMar} and \cite{HMUT}).

\subsection{Acknowledgements} 

We would like to thank Steve Hofmann for suggesting this problem to us, for our discussions with him that illuminated the dyadic analysis part of the paper and in general, for his constant support. The second named author has benefited from a two-year Sophie Germain International post-doctoral scholarship in Fondation de Math\'{e}matiques Jacques Hadamard (FMJH) and would like to thank Universit\'{e} Paris-Sud 11, Orsay for its hospitality. This work started when the second named author was a research assistant under Steve Hofmann in the Department of Mathematics at the University of Missouri, Columbia.

\section{Notation and Preliminaries}

\begin{list}{\labelitemi}{\leftmargin=1em}
\item We shall use the letters $c$, $C$ to denote positive constants, not necessarily the same at each occurrence, which depend only on dimension and the constants appearing in the hypotheses of the theorems. We shall also write $A \lesssim B$ and $A \approx B$ to mean, respectively, that $A \leq C B$ and $0 < c \leq A/B \leq C$, where the constants $c$ and $C$ are as above, unless explicitly noted.
\item Given a domain $\Omega \subset \mathbb{R}^{n+1}$, we shall use lower case letters $x$, $y$, $z$, etc., to denote points on $\partial \Omega$, and capital letters $X$, $Y$, $Z$, etc., to denote generic points in $\mathbb{R}^{n+1}$ (especially those in  $\mathbb{R}^{n+1} \setminus \partial \Omega$).
\item For a Borel set $A \subset \mathbb{R}^{n+1}$, we let $1_A$ denote the usual indicator function of $A$, i.e. $1_A(x) = 1$ if $x \in A$, and $1_A(x) = 0$ if $x \notin A$.
\item We let $\sigma$ be the restriction of the $n$-dimensional Hausdorff measure to $\partial \Omega$, i.e., $\sigma:=\mathcal{H}^n\mid_{\partial \Omega}$.
\item For $X \in \mathbb{R}^{n+1}$, we set $\delta(X) := dist(X, \partial \Omega)$.
\item The open $(n+1)$-dimensional Euclidean ball of radius $r$ will be denoted $B(x,r)$ when the center $x$ lies on $\partial \Omega$, or $B(X,r)$ when the center $X \in \mathbb{R}^{n+1} \setminus \partial \Omega$. A “surface ball” is denoted $\Delta(x,r) := B(x,r) \bigcap \partial \Omega$.
\end{list}

\begin{definition}
An $n$-dimensional set $E\subseteq\mathbb{R}^{n+1}$ is {\tt Ahlfors-David regular} and we write $E \in ADR$, if $E$ is closed and for every ball $B(x,r)$ centered at $x \in E$ and of radius $r \in \left(0,diam(E)\right)$,
\begin{equation} C^{-1}r^n \leq \sigma(B(x,r))\leq C r^n,\label{Not.ADreg}\end{equation} for some positive constant $C$.
\end{definition}

\begin{remark}

Throughout the paper we shall consider $\Omega=\mathbb{R}^{n+1}\setminus E$, $E=\partial \Omega\in ADR$. Observe that $\Omega$ need not be connected.
\end{remark}

The following lemma shows that for $ADR$ spaces or even for spaces for homogeneous-type one can build a family of subsets of $E$ that play the same role that dyadic cubes do in the Euclidean space. For the construction of such ``cubes" see for example \cite{Ch}, \cite{D1}, \cite{HytK} and \cite{AHyt}.

\begin{lemma}[``{\bf Dyadic grid}" {\bf in $ADR$ spaces}]\label{Not.Dyad.Cubes} Let $E\subset \mathbb{R}^{n+1}$ be in $ADR$. Then there exist constants $\alpha_0 > 0$, $\eta > 0$ and $C_1, C_2 < \infty$, depending only on dimension and the $ADR$ constants, such that for each $k\in \mathbb{Z}$, there is a collection of Borel sets (``cubes")
\begin{equation*}\mathbb{D}_k:= \{ Q_j^k \subset E:j \in \mathcal{I}_k \},\end{equation*}
where $\mathcal{I}_k$ denotes some (possibly finite) index set depending on $k$, satisfying
\begin{itemize}
\item[(1)] $ \ E = \displaystyle\bigcup_{j \in \mathcal{I}_k} Q_j^k$ for each $k \in \mathbb{Z}$.

\item[(2)] If $k < k'$, then either $Q^{k'}_{\ell} \bigcap Q^k_j=\emptyset$ or $Q^{k'}_{\ell} \subset Q^{k}_{j}$, $\forall \ell\in\mathcal{I}_{k'}$, $\forall j\in\mathcal{I}_k$.
\item[(3)] For each $(l,k')$ and each $k<k'$, there is a unique $j$ such that $Q^{k'}_{\ell} \subset Q^{k}_{j}$.
\item[(4)] If $Q^k_j\in\mathbb{D}_k$ then 
\begin{align*} C_1^{-1}2^{-k} &\leq diam(Q^k_j)\leq C_1 2^{-k}\\
C_1^{-1}2^{-kn} &\leq \mathcal{H}^n(Q^k_j)\leq C_1 2^{-kn}.\end{align*}
\item[(5)] Each $Q^k_j$ contains some ``surface ball" $\Delta(x^k_j,\alpha_0 2^{-k}):=B(x^k_j,\alpha_0 2^{-k}) \bigcap E$.
\item[(6)] For all $k, j$ and for all $\tau \in (0,\alpha_0)$,
$$\mathcal{H}^n\left( \left\{x \in Q^k_j:dist(x,E \setminus Q^k_j)\leq \tau 2^{-k} \right\}\right) \leq C_2 \tau^{\eta}\mathcal{H}^n(Q^k_j).$$
\end{itemize}
\end{lemma}

Let us now mention few remarks concerning the last lemma.

\begin{remark}
\begin{itemize}
\item[(i)] We shall denote by $\mathbb{D}=\mathbb{D}(E)$ the collection of all relevant $Q^k_j$, i.e.,
$$ \mathbb{D} =\bigcup_k \mathbb{D}_k.$$
\item[(ii)] If $diam(E)$ is finite we ignore all those $k \in \mathbb{Z}$ such that $2^{-k} \gtrsim diam(E)$ and thus previously the union runs over all those $k$'s for which $2^{-k} \lesssim diam(E)$.
\item[(iii)]For a dyadic cube $Q \in \mathbb{D}_k$ we set $\ell(Q)=2^{-k}$ and we shall refer to this quantity as the ``length'' of $Q$. Obviously, $\ell(Q) \approx diam(Q)$.
\end{itemize}
\end{remark}

\begin{definition}
In view of \cite{HSw} and \cite{HMY}, a sequence of operators $\{S_k\}_{k \in \mathbb{Z}}$ is said to be an {\tt approximation to the identity} if there exist $0<\epsilon$ and $0<A,C< \infty$ such that for all $x, x', y$ and $y' \in E$, the kernel of $S_k$, are functions from $E \times E$ into $\mathbb{C}$ satisfying

\begin{align}
S_j(x,y)&=0,\,\, \text{if}\,\, |x-y|\geq C2^{-j},\label{Sj.compact.support}\\
|S_j(x,y)|&\leq C 2^{jn},\label{Sj.pointwise}\\
|S_j(x,y)-S_j(x',y)|&\leq C 2^{j(n+\epsilon)}|x-x'|^{\epsilon},\label{Sj.Holder1}\\
|S_j(x,y)-S_j(x,y')|&\leq C 2^{j(n+\epsilon)}|y-y'|^{\epsilon},\label{Sj.Holder2}\\
\int S_j(x,y) d\sigma(x) &=1, \label{Sj.int=0.x} \\
\int S_j(x,y) d\sigma(y) &=1.\label{Sj.int=0.y}
\end{align}
\end{definition}
Notice that the bound in \eqref{Sj.pointwise} can be replaced by $\frac{C 2^{-j\epsilon}}{(2^{-j}+|x-y|)^{n+\epsilon}}$. Indeed, since $S_j$ is supported in the set $\{|x-y|< C 2^{-j}\}$, it is always true that
$$2^{jn} \approx \frac{2^{-j\epsilon}}{(2^{-j}+|x-y|)^{n+\epsilon}}.$$
\begin{properties}
If we set $D_j=S_j-S_{j-1}$ then obviously, the kernel of $D_j$, say $\varphi_j(x,y)$, satisfies also properties \eqref{Sj.Holder1} and \eqref{Sj.Holder2} and in place of \eqref{Sj.int=0.x} and \eqref{Sj.int=0.y}, 
$$\int_E D_j(x,y) d\sigma(x)=\int_E D_j(x,y) d\sigma(y)=0.$$
Moreover, there exists a family of operators $(\widetilde{D}_j)_j$ (see \cite{HSw} for details) such that a discrete Calder\'on-type reproducing formula holds, i.e.,
\begin{equation}
I=\displaystyle\sum_j D_j\widetilde{D}_j , \,\,\,\text{in the}\,\, L^2 \,\, \text{sense},
\label{C-Z.reproducing} \end{equation}
and also
\begin{equation}\label{bound.tilde.Dj}
\int_E \left(\sum_j |\widetilde{D}_j f(x)|^2\right)^{p/2}\!\! d\sigma(x) \lesssim \int_E |f(x)|^p d\sigma(x) , \,\text{for}\, 1<p<\infty.
\end{equation}
\end{properties}

\begin{definition}
We shall introduce now some notation. Let $\Omega:= \mathbb{R}^{n+1} \setminus E$ be as before and $I$ be an $n+1$-dimensional cube, then given a ``surface" cube $Q$, we define an associated ``{\tt Whitney region}" as follows: Let $\mathcal{W} = \mathcal{W}(\Omega)$ denote a collection of dyadic Whitney cubes of $\Omega$, so that the cubes in $\mathcal{W}$ form a pairwise non-overlapping covering of $\Omega$, which satisfy
\begin{equation}
4 diam(I) \leq dist(4I, \partial \Omega) \leq dist(I, \partial \Omega) \leq 40 diam(I),\quad	\text{for every}\quad I \in \mathcal{W},
\end{equation} 
(just divide dyadically the standard Whitney cubes, as constructed in \cite{St}, Chapter
VI, into cubes with side length $1/8$ as large) and also 
\begin{equation}
(1/4) diam(I_1) \leq diam(I_2) \leq 4 diam(I_1),
\end{equation}

\noindent whenever $I_1$ and $I_2$ ``touch". Since $E$ is closed, $\Omega$ is open and so it allows a Whitney decomposition $\Omega=\displaystyle\bigcup_{I\in \mathcal{W}}I$ where 
$$\mathcal{W}:=\bigcup_k \mathcal{W}_k \quad \text{and} \quad \mathcal{W}_k:=\{I:\ell(I) = 2^{-k}\}.$$
\end{definition}

\begin{definition}
Fix $x\in E$. Then we define the {\tt cone} $\Gamma_{\beta}\,(x)$ with vertex $x$ and aperture $\beta>0$ to be
\begin{equation}\label{def.cone}\Gamma_{\beta}\,(x):=\bigcup_{Q \ni x} \mathcal{U}_{Q,\,\beta},\end{equation}
where 
$$\mathcal{U}_{Q,\,\beta}=\bigcup_{I\in\mathcal{C}_{Q,\,\beta}}I,$$ 
and  
$$\mathcal{C}_{Q,\,\beta}:=\left\{I\in \mathcal{W} : \ell(I)/8 \leq \ell(Q) \leq 8 \ell(I),\,\, dist(Q,I)\leq\beta \ell(Q) \right\}.$$ 
The {\tt $Q$-truncated cone} will be denoted by
$$\Gamma_{Q,\,\beta}\,(x):= \bigcup_{\substack{Q'\ni x\\ Q' \subseteq Q}} \mathcal{U}_{Q',\,\beta}.$$
\end{definition}

Observe that, by choosing $\beta$ sufficiently large, depending on the Ahlfors-David parameters, we may suppose that $\mathcal{C}_{Q,\,\beta}\neq\emptyset$, for every $Q$. We fix now such a $\beta$ which we use in the sequel and for simplicity we write $\Gamma(x)$, $\Gamma_Q(x)$, $\mathcal{C}_Q$ and $\mathcal{U}_Q$.

\begin{definition}
We also define 
\begin{equation}
\Theta f(X):=\int_E\psi(X,y)f(y)d\sigma(y),
\label{Theta}\end{equation}
where for $\psi(X,y)$, there exist $C>0$ and $\alpha>0$, such that
\begin{align}
|\psi(X,y)| &\leq C\frac{\delta(X)^{\alpha}}{|X-y|^{n+\alpha}}, \quad y\in E,\, X\in\Omega, \label{psi.ptwise.estimate}\\
|\psi(X,y)-\psi(X,y')| &\leq C \frac{|y-y'|^{\alpha}}{|X-y|^{n+\alpha}},\quad y, y' \in E,\,X\in \Omega, \,\, 2|y-y'| \leq |X-y|.
\label{psi.holder.cont}\end{align}
\label{definitiontheta}\end{definition}
Our main goal is to prove the following theorem.
\begin{theorem}
Let $E \in ADR$ and $\Theta$ be as in definition \ref{definitiontheta}. Suppose also that there exists a positive constant $C_0$, $p \in (1,\infty)$ and a system $\{b_Q\}$ of funtions indexed by dyadic cubes $Q\subseteq E$, such that for each dyadic cube $Q \in \mathbb{D}(E)$,

\begin{align} 
&\left|\int_Q b_Q(x)d\sigma(x) \right| \geq \frac{1}{C_0}\sigma(Q) \label{main.theorem.eq1}\\
&\int_E \left|b_Q(x) \right|^p d\sigma(x) \leq C_0\sigma(Q)\label{main.theorem.eq2} \\
&\int_{Q}\left(\iint_{\Gamma_Q(x)}|\Theta b_Q(Y)|^2\frac{dY}{\delta(Y)^{n+1}} \right)^{p/2}d\sigma(x) \leq C_0\sigma(Q). \label{main.theorem.eq3}
\end{align}
Then there exists a positive constant $C$ such that the following square function estimate holds.
\begin{equation}
\iint_{\Omega}|\Theta f(Y)|^2\frac{dY}{\delta(Y)}\leq C \int_E|f(x)|^2 d\sigma(x).
\label{square.function.estimate}\end{equation}
\label{main.theorem}\end{theorem}
\begin{remark} Notice that in the case $p=2$, \eqref{main.theorem.eq3} is equivalent to $$\iint_{\mathcal{T}_Q}|\Theta b_Q(Y)|^2\frac{dY}{\delta(Y)}\leq C_0\sigma(Q),$$
where $\mathcal{T}_Q=\displaystyle\bigcup_{Q'\subseteq Q}\displaystyle\mathcal{U}_{Q'}$.
\end{remark}

In section \ref{sec.3} we prove our results in the case that $E$ is unbounded and the proof of the theorem goes through the proof of an auxiliary lemma (Lemma \ref{c3.lemma}) and a $T1$ theorem for square functions, while in section \ref{sec.4}, we show that the same result holds even in the case that $E$ is bounded.

\section{Proof of Theorem \ref{main.theorem} when $E$ is an unbounded $ADR$ set} \label{sec.3}

To prove Theorem \ref{main.theorem} we shall need the following lemma.

\begin{lemma}\label{c3.lemma}
Suppose that there are $\eta\in(0,1)$ and $C_1>0$ such that for every cube $Q \in \mathbb{D}(E)$ there exists a family $\mathcal{F}=\{Q_k\}$ of dyadic subcubes of $Q$ such that for any $p \in (1,2]$,
\begin{align}
&\sum_k \sigma(Q_k) \leq (1-\eta)\sigma(Q), \label{c3.theorem.eq1}\\
&\int_Q \left(\iint_{\gamma_Q(x)} \left|\Theta 1(Y)\right|^2\frac{dY}{\delta(Y)^{n+1}}\right)^{p/2}d\sigma(x) \leq C_1 \sigma(Q) \label{c3.theorem.eq2},
\end{align}
where
$$ \gamma_Q(x):=\bigcup_{\substack{Q' \ni x \\Q' \in Good(Q)}} \mathcal{U}_{Q'}$$ 
and 
\begin{align*} Good(Q) := \{ Q' \subseteq Q :\,\, & \text{either} \,\, Q'\cap Q_k=\emptyset,\,\, \text{for every}\,\, k,\\
&\text{or if}\,\, Q_k \cap Q' \neq \emptyset,\,\, \text{for some}\,\,k ,\,\, \text{then}\,\, \ell(Q')>\ell(Q_k) \}.
 \end{align*}

Then, there exists a positive constant $C$ such that
\begin{equation}
\sup_Q \frac{1}{\sigma(Q)}\int_Q \iint_{\Gamma_Q(x)}\left|\Theta1(Y)\right|^2\frac{dY}{\delta(Y)^{n+1}}d\sigma(x) < C.
\end{equation}

\end{lemma}

\begin{remark}
The ``good" cubes are those $Q'\subseteq Q$ that are not contained in any $Q_k\in\mathcal{F}$ and $Good(Q)$ is sometimes called the discrete sawtooth relative to the family $\mathcal{F}$.
\end{remark}

Let us give an outline of the proof. We first show that the conditions of Theorem \ref{main.theorem} imply the conditions of Lemma \ref{c3.lemma}. Then we prove Lemma \ref{c3.lemma} and by the use of a $T1$ Theorem for square functions that we show at the end, we conclude Theorem \ref{main.theorem}. 

\subsection{Conditions of Theorem \ref{main.theorem} imply conditions of Lemma \ref{c3.lemma}}\label{3.1}

Fix $Q\in\mathbb{D}(E)$ and then for this cube we choose the family $\mathcal{F}=\{Q_k\}$ as subcubes of Q maximal with respect to the stopping time condition $\mathcal{R}e\int_{Q_k}b_Q(x)d\sigma(x)\leq{C_0}^{-1}\sigma(Q_k)$; thus, we have that $|\mathbb{E}_{Q'}b_Q|\geq{C_0}^{-1}$, for every $Q'\in Good(Q)$, where $$\mathbb{E}_{Q'}f:=\frac{1}{\sigma(Q')}\int_{Q'}f(x) d\sigma(x).$$

Then the proof of \eqref{c3.theorem.eq1} follows from \eqref{main.theorem.eq1} and \eqref{main.theorem.eq2} by following {\it mutatis mutandi} the corresponding argument in \cite[p.4]{H1}, so it is enough to check the validity of \eqref{c3.theorem.eq2}.

\begin{align}
&\int_Q \left(\iint_{\gamma_Q(x)} \left| \Theta1(Y) \right|^2 \frac{dY} {\delta(Y)^{n+1}}\right)^{p/2}d\sigma(x) \notag \\ 
&= \int_Q\left(\sum_{\substack{Q' \ni x\\ Q'\in Good(Q)}} \iint_{\mathcal{U}_{Q'}} \left|\Theta1(Y)\right|^2\frac{dY}{\delta(Y)^{n+1}}\right)^{p/2}d\sigma(x) \notag \\
&\leq C \int_Q \left(\sum_{\substack{Q' \ni x \\ Q'\in Good(Q)}} \iint_{\mathcal{U}_{Q'}}\left|\Theta 1(Y)\right|^2 \left| \mathbb{E}_{Q'}b_Q \right|^2\frac{dY}{\delta(Y)^{n+1}}\right)^{p/2}d\sigma(x),\label{c3.Cond.Th=>Cond.Lemma.Good.Sq.F.Est.}
\end{align}
and it remains to prove that the latter is less than a constant multiple of $\sigma(Q)$. At this point we shall use the Coifman-Meyer method and decompose 
\begin{equation}(\Theta1) \mathbb{E}_{Q'}=(\Theta1 \mathbb{E}_{Q'}-\Theta )+  \Theta =: R +  \Theta, \label{C-M.split.p=2}
\end{equation}

\noindent and by \eqref{main.theorem.eq3} we only need to prove that
\begin{equation*}
\int_Q\left( \sum_{\substack{Q' \ni x \\ Q'\in Good(Q)}} \iint_{\mathcal{U}_{Q'}}\left| Rb_Q(Y)\right| ^2\frac{dY}{\delta(Y)^{n+1}}\right)^{p/2}d\sigma(x) \lesssim \| b_Q \|_{L^p}^p \lesssim \sigma(Q),
\end{equation*}
or equivalently,
\begin{equation}
\int_Q \left(\sum_k \sum_{\substack{ Q' \in \mathbb{D}_k\\Q'\in Good(Q)}} \textbf{1}_{Q'}(x)\iint_{\mathcal{U}_{Q'}} \left| Rb_Q(Y)\right| ^2\frac{dY}{\delta(Y)^{n+1}}\right)^{p/2}d\sigma(x) \lesssim  \| b_Q\|_{L^p}^p \lesssim \sigma(Q).\label{c3.Cond.Th=>Cond.Lemma.Good.Sq.F.Est.Rf}
\end{equation}

By the discrete Calder\'{o}n's reproducing formula \eqref{C-Z.reproducing} one can write
\begin{align*} 
Rb_Q(Y) &=\sum_j \left( \Theta 1(Y) \mathbb{E}_{Q'}- \Theta \right) (D_j \widetilde{D}_j) b_Q(Y)\\ &=\sum_j \left( \Theta1(Y) \mathbb{E}_{Q'}-\Theta \right)D_j f_j(Y),
\end{align*}
where $f_j :=\widetilde{D}_j b_Q$. If $\varphi_j(\cdot ,\cdot)$ is the kernel of $D_j$, the latter part of the equation above is equal to
$$\sum_j\int_E f_j(w) T_j(w,Y) \, d\sigma(w),$$
where
$$ T_j(w,Y):=\int_E \left( \frac{\Theta 1(Y)}{\sigma(Q')} \textbf{1}_{Q'}(z)-\psi(Y,z) \right) \varphi_j(z,w)\,d\sigma(z).$$

Let us recall that the diameter of the cube $Q'$ is $2^{-k}$ since $Q'\in\mathbb{D}_k$ and decompose the problem in two cases. We first let the indices $k, j$ to be so that $2^{-k} \leq C_2 2^{-j}$ and $x$ be an arbitrary point of $Q'$. Since $\int_E \frac{\Theta1(Y) }{\sigma(Q')} \textbf{1}_{Q'}(z)-\psi(Y,z)d\sigma(z)=0$ we can subtract $\varphi_j(x,w)$ for free and thus $T_j(w,Y)$ is equal to 
\begin{equation*}
\int_E \left( \frac{\Theta1(Y) }{\sigma(Q')} \textbf{1}_{Q'}(z)-\psi(Y,z) \right) \left( \varphi_j(z,w)-\varphi_j(x,w)\right)\, d\sigma(z).
\end{equation*}
We shall now write the domain of integration as the union of its subsets 
$$E_1:=\left\{z \in E: |z-x| \leq C_3 2^{-(k+j)/2}\right\}\,\,\, \text{and}\,\,\, E_2:=\left\{z \in E: |z-x| > C_3 2^{-(k+j)/2}\right\},$$ 
where $C_2,C_3>0$ are absolute constant appropriately chosen later on. Therefore, $T_j(w,Y)$ can be written as the sum $T^1_j(w,Y)+T^2_j(w,Y)$, where
$$T^i_j(w,Y):= \int_{E_i} \left( \frac{\Theta1(Y)}{\sigma(Q')} \textbf{1}_{Q'}(z)-\psi(Y,z) \right) \left( \varphi_j(z,w)-\varphi_j(x,w)\right)\, d\sigma(z),$$
for $i=1,2$. By H\"{o}lder's continuity of $\varphi_j$ and simple geometric considerations one can see that 
$$ |\varphi_j(z,w)-\varphi_j(x,w)| \lesssim \frac{|x-z|^{\epsilon}}{(2^{-j} +|x-w|)^{n+\epsilon}}.$$
The latter, along with the bound $|\Theta1(Y)| \lesssim\int_E\frac{2^{-k\alpha}}{(2^{-k}+|x-z|)^{n+\alpha}} d\sigma(z)\lesssim 1$, the pointwise bounds of $\psi$ and that $\delta(Y) \approx 2^{-k}$ and $|Y-z| \approx 2^{-k}+|x-z|$, since $x\in Q'$ and $Y\in\mathcal{U}_{Q'}$, shows that
\begin{align*}
|T^1_j(w,Y)| &\lesssim  \int_{E_1} \frac{2^{-k \alpha}}{(2^{-k} +|x-z|)^{n+\alpha}} \frac{|x-z|^{\epsilon}}{(2^{-j} +|x-w|)^{n+\epsilon}}\,d\sigma(z)\\
&+\frac{1}{\sigma(Q')}\int_{Q'} \frac{|x-z|^{\epsilon}}{(2^{-j} +|x-w|)^{n+\epsilon}}\,d\sigma(z)\\
&\lesssim  \int_{E_1} \frac{2^{-k \alpha}}{(2^{-k} +|x-z|)^{n+\alpha}} \frac{2^{-\epsilon(k+j)/2}}{(2^{-j} +|x-w|)^{n+\epsilon}}\,d\sigma(z)\\
&+\frac{1}{\sigma(Q')}\int_{Q'} \frac{2^{-\epsilon(k+j)/2}}{(2^{-j} +|x-w|)^{n+\epsilon}}\,d\sigma(z)\\
& \lesssim \left(\frac{2^{-k}}{2^{-j}}\right)^{\epsilon/2} \frac{2^{-j\epsilon}}{(2^{-j} + |x-w|)^{n+\epsilon}}=:\left(\frac{2^{-k}}{2^{-j}}\right)^{\epsilon/2} \omega_{j,\epsilon}(x,w),
\end{align*}
where in the last inequality we used the fact that 
$\int_E  \omega_{k,\alpha} (x,z)\, d\sigma(z) \lesssim 1.$
Moreover, if we define
$$T^{2,1}_j(w,Y):= \int_{E_2}  \frac{\Theta1(Y)}{\sigma(Q')} \textbf{1}_{Q'}(z) \left( \varphi_j(z,w)-\varphi_j(x,w)\right)\, d\sigma(z)$$
and
$$T^{2,2}_j(w,Y):= \int_{E_2}\psi(Y,z)  \left( \varphi_j(z,w)-\varphi_j(x,w)\right)\, d\sigma(z),$$
then $T^2_j(w,Y)=T^{2,1}_j(w,Y)-T^{2,2}_j(w,Y)$. It is not hard to see that $T^{2,1}_j(w,Y)=0$. Indeed, if $x\in Q'$ and $l(Q')\approx 2^{-k}$, then $|x-z|>C_32^{-(k+j)/2}>C_32^{-k}$, and thus we can choose $C_3$ big enough so that $z\not\in Q'$, which proves our claim. To this end, it is only left to bound $T^{2,2}_j(w,Y)$. For, if we set $\widetilde{\omega}_j(x,y):= \textbf{1}_{\{|x-y|\lesssim 2^{-j}\}} 2^{jn}$, obviously, $|\varphi_j(x,y)| \lesssim \widetilde{\omega}_j(x,y)$, which in turn shows that
\begin{align*}
|T^{2,2}_j(w,Y)| &\lesssim \int_{E_2} \omega_{k,\alpha} (x,z) \left(\widetilde{\omega}_j(z,w)+\widetilde{\omega}_j(x,w) \right)d\sigma(z)\\
&\lesssim \left(\frac{2^{-k}}{2^{-j}}\right)^{\alpha/2} \int_{E_2} \omega_{\frac{k+j}{2},\alpha} (x,z) \left(\widetilde{\omega}_j(z,w)+\widetilde{\omega}_j(x,w) \right)d\sigma(z).\end{align*}
Clearly, 
$$\int_E \widetilde{\omega}_j(x,w) |f_j(w)|\, d\sigma(w) \lesssim \mathcal{M}(f_j)(x),$$
where $\mathcal{M}$ stands for the dyadic Hardy-Littlewood maximal operator, which after applying Fubini conveniently above, allows us to conclude that \begin{equation}\label{c3.estimate.k>j}
\left| \int_E  f_j(w) T_j(w,Y)\, d\sigma(w) \right| \lesssim \left(\frac{2^{-k}}{2^{-j}}\right)^{\min\{\alpha,\epsilon\}/2} \Big( \mathcal{M}(\mathcal{M}(f_j))(x)+\mathcal{M}(f_j)(x) \Big).
\end{equation} 

We now turn our attention to the case that $2^{-j}<\frac{1}{C_2} 2^{-k}$. We treat $T_j(w,Y)$ by writing it as the difference 
$$\int_E  \Theta1(Y) \frac{1}{\sigma(Q')} \textbf{1}_{Q'}(z) \varphi_j(z,w)\,d\sigma(z)-\int_E  \psi(Y,z) \varphi_j(z,w)\,d\sigma(z)=:I^1_j(w,Y)-I^2_j(w,Y)$$
and then we bound each part separately. The main difficulty relies on estimating $I^1_j(w,Y)$ while it is fairly easy to derive the desired bound for $I^2_j(w,Y)$. Indeed, by choosing $C_2$ depending on the dimension and the support of $\varphi_j$ so that $2|z-w|\leq |Y-w|$, we have
\begin{align*}
\left| \int_E  \psi(Y,z) \varphi_j(z,w)\,d\sigma(z) \right| &= \left| \int_E   \left( \psi(Y,z) - \psi(Y,w) \right) \varphi_j(z,w)\,d\sigma(z) \right|\\
&\lesssim\int_E \frac{|z-w|^\alpha}{\left(2^{-k}+ |x-w| \right)^{n+\alpha}} |\varphi_j(z,w)|\,d\sigma(z).
\end{align*}
Since $\varphi_j(z,w)$ is supported in the ball $\{|z-w|<C 2^{-j}\}$, the latter is bounded by a constant multiple of
\begin{equation*} \int_E \frac{2^{-j\alpha}}{\left(2^{-k}+ |x-w| \right)^{n+\alpha}} |\varphi_j(z,w)|\,d\sigma(z),\end{equation*}
which, in turn, in glance of $\int_E|\varphi_j(z,w)|\,d\sigma(z)\lesssim 1$, shows that 
\begin{equation} |I^2_j(w,Y)| \lesssim \left(\frac{2^{-j}}{2^{-k}}\right)^\alpha \omega_{k,\alpha}(x,w).\label{c3.estimate1.j>>k}\end{equation}
Hence,
$$ \left|\int_E f_j(w) I^{2}_j(w,Y) \,d\sigma(w) \right| \lesssim \left(\frac{2^{-j}}{2^{-k}}\right)^\alpha \mathcal{M}(f_j)(x).$$

In order to handle $I^1_j(w)$ we shall introduce an auxiliary function $\eta_{Q'}\!\in\!\mathcal{C}_0^{\infty}(Q')$  that vanishes outside $Q'$ and is identically $1$ in a set $R \subset Q'$, which set has the additional property 
$$\sigma(Q' \setminus R) \approx 2^{-(j+k)/2} 2^{-k(n-1)}.$$ 
Moreover, we may assume that $\eta_{Q'}$ satisfies the estimate
$$||\nabla\eta_{Q'}||_{L^{\infty}}\lesssim \frac{1}{ 2^{-(j+k)/2}}.$$
Before we show that we can construct such a function let us see how it helps us bound $I^1_j(w,Y)$. By adding and subtracting $\eta_{Q'}$ in $I^1_j(w,Y)$ the latter is equal to
\begin{align*}
\frac{1}{\sigma(Q')} \int_E  \Theta1(Y)  ( \textbf{1}_{Q'} - \eta_{Q'})(z) \varphi_j(z,w)\,d\sigma(z) &+\frac{1}{\sigma(Q')}  \int_E  \Theta1(Y) \eta_{Q'}(z) \varphi_j(z,w)\,d\sigma(z)\\
&=:I^{1,1}_j(w,Y)+I^{1,2}_j(w,Y).
\end{align*}
Since $\int_E \varphi_j(z,w)d\sigma(z)=0$ and $\varphi_j(z,w)=0$ for $|z-w|>C2^{-j}$, by the mean value theorem,
\begin{align*} | I^{1,2}_j(w) |&=\left| \frac{1}{\sigma(Q')}\int_E  \Theta1(Y) (\eta_{Q'}(z)-\eta_{Q'}(w))\varphi_j(z,w) \,d\sigma(z) \right|\\
&\lesssim\frac{2^{(j+k)/2}}{\sigma(Q')}\textbf{1}_{\widetilde{Q}'}(w) \int_E |z-w| |\varphi_j(z,w)| \,d\sigma(z) \\
&\lesssim\frac{2^{(-j+k)/2}}{\sigma(Q')}\textbf{1}_{\widetilde{Q}'}(w),
\end{align*}
\noindent where $\widetilde{Q}'$ is a ``fattened version" of $Q'$ in E.
Therefore, using that $\sigma(Q')\approx\sigma(\widetilde{Q}')$,
$$ \left|\int_E f_j(w) I^{1,2}_j(w) \,d\sigma(w) \right| \lesssim 2^{(-j+k)/2} \mathcal{M}(f_j)(x).$$

On the other hand, by Fubini and the size estimates of $Q'\!\setminus\! R$,
\begin{align*} \left|\int_E f_j(w) I^{1,1}_j(w) \,d\sigma(w) \right| &\lesssim 
\frac{1}{\sigma(Q')}\int_E \int_{Q' \setminus R} |\varphi_j(z,w)| |f_j(w)|\,d\sigma(z)\,d\sigma(w)\\
&\lesssim\frac{1}{\sigma(Q')} \int_{Q' \setminus R} \int_E \widetilde{\omega}_j(z,w)|f_j(w)|\,d\sigma(w)\,d\sigma(z)\\
&\lesssim \frac{1}{\sigma(Q')} \int_{Q' \setminus R} \mathcal{M}(f_j)(z)\,d\sigma(z),
\end{align*}
which by H\"{o}lder's inequality for $1<r<p\leq2$ is bounded by
$$2^{(-j+k)/2r'}  \Big( \mathcal{M} \left(\mathcal{M}(f_j)\right)^r\Big)^{1/r}(x),$$
where $r'$ is the H\"{o}lder conjugate of $r$. Putting everything together we can conclude that there exists $\beta_0\geq \min\{\frac{1}{2r'},\frac{1}{2},\frac{\alpha}{2},\frac{\epsilon}{2}\}>0$ such that
\begin{equation*}
\left|\int_E f_j(w) T_j(w) \,d\sigma(w) \right| \lesssim 2^{-\beta_0|k-j|} \left(  \mathcal{M}(f_j) + \mathcal{M} \left(\mathcal{M}(f_j)\right) + \Big(\mathcal{M} \left(\mathcal{M}(f_j)\right)^r\Big)^{1/r} \right)(x),
\end{equation*}
\noindent uniformly for every $x\in Q'$.

Notice that the right part of the last inequality is controlled by $C 2^{-\beta_0|k-j|}  \Big(\mathcal{M} \left(\mathcal{M}(f_j)\right)^r\Big)^{1/r}(x)$, and thus, the left part of inequality \eqref{c3.Cond.Th=>Cond.Lemma.Good.Sq.F.Est.Rf} is

\begin{align*}
& \lesssim\int_Q\left(\sum_{k}\sum_{Q'\in\mathbb{D}_k}\textbf{1}_{Q'}(x)\iint_{\mathcal{U}_{Q'}}\left(\sum_j2^{-\beta_0|k-j|} \Big(\mathcal{M} \left(\mathcal{M}(f_j)\right)^r\Big)^{1/r}(x) \right)^2\frac{dY}{\delta(Y)^{n+1}}\right)^{p/2}\!\!\!\!\! d\sigma(x)\\
&\lesssim\int_Q\left(\displaystyle\sum_{k}\sum_{Q'\in\mathbb{D}_k}\textbf{1}_{Q'}(x) \left(\sum_j2^{-\beta_0|k-j|} \Big(\mathcal{M} \left(\mathcal{M}(f_j)\right)^r\Big)^{1/r}(x)\right)^2\right)^{p/2}d\sigma(x)\\
&\lesssim\int_Q\left(\sum_k\sum_{Q'\in\mathbb{D}_k}\textbf{1}_{Q'}(x)\sum_j 2^{-\beta_0|k-j|}\Big(\mathcal{M} \left(\mathcal{M}(f_j)\right)^r\Big)^{2/r}(x)\right)^{p/2}d\sigma(x)\\
&\approx\int_Q\left(\sum_k\sum_j 2^{-\beta_0|k-j|}\Big(\mathcal{M} \left(\mathcal{M}(f_j)\right)^r\Big)^{2/r}(x)\right)^{p/2}d\sigma(x)\\
&\approx\int_Q\left(\sum_j\Big(\mathcal{M} \left(\mathcal{M}(f_j)\right)^r\Big)^{2/r}(x)\right)^{p/2}d\sigma(x)
\end{align*}

At this point we shall make use of the inequalities for the Hardy-Littlewood maximal operator that were proved in \cite{FS} (the techniques to prove the result also apply to our setting) or in a more general context in \cite{GCRDF}, which in conjunction with \eqref{bound.tilde.Dj} and \eqref{main.theorem.eq2} allows to obtain our result.

\begin{align*}
&\int_Q\left(\sum_j\Big(\mathcal{M} \left(\mathcal{M}(f_j)\right)^r\Big)^{2/r}(x)\right)^{p/2}d\sigma(x) \lesssim \int_E \left(\sum_j\left(\mathcal{M}(f_j)\right)^2(x)\right)^{p/2}d\sigma(x)\\
&\lesssim\int_E\left(\sum_j|f_j|^2(x)\right)^{p/2}d\sigma(x)=\int_E\left(\sum_j|\widetilde{D}_jb_Q(x)|^2\right)^{p/2}d\sigma(x)\\
&\lesssim\int_E|b_Q(x)|^pd\sigma(x) \lesssim \sigma(Q),\end{align*}
\noindent where in the last step we have used \eqref{bound.tilde.Dj}. This proves \eqref{c3.Cond.Th=>Cond.Lemma.Good.Sq.F.Est.Rf} and consequently \eqref{c3.theorem.eq2}.
\\

It only remains to prove our earlier claim and show that there exists $\eta_{Q'}$ with the aforementioned properties. For, we cover $Q'$ by non-overlapping dyadic cubes $Q'' \in \mathbb{D}_m(Q')$ where $2^{-m} \approx 2^{-(k+j)/2}$ and we remove those that satisfy the condition $dist(Q'', E\setminus\! Q') \leq (C_1)^2 2^{-m}$. We shall say that the remaining ones are in the collection $\widetilde{\mathbb{D}}_m(Q')$. We choose $C_1$ large enough depending on the constant $\alpha_0>0$ defined in Lemma \ref{Not.Dyad.Cubes} so that 

\begin{align*}
Q'' &\subset B(x_{Q''}, \alpha_0 \ell(Q'')) \cap E\\
&\subset B(x_{Q''}, 2\alpha_0 \ell(Q'')) \cap E=:B_{Q''} \cap E\\
& = \Delta(x_{Q''}, 2\alpha_0 \ell(Q''))=: \Delta_{Q''}
\end{align*}

\noindent and also $dist(\Delta_{Q''}, E\setminus\! Q') \geq C_1 2^{-m}$ and $dist(B_{Q''}, E\setminus\! Q') \geq C_1 2^{-m}$ for all $Q''\in\widetilde{\mathbb{D}}_m(Q')$. We now take a partition of unity $\left\{\Phi_{Q''}\right\}$ subordinate to the family of the surface balls $\{ \Delta_{Q''}\}_{Q'' \in \mathbb{D}_m(Q')}$ (i.e., we take $\left\{\Phi_{Q''}\right\}$ subordinate to the family of balls $\{B_{Q''}\}_{Q'' \in \mathbb{D}_m(Q')}$ and then we restrict the domain to E) so that if $R := \left(\overline{ \bigcup_{Q'' \in \widetilde{\mathbb{D}}_m(Q')} \Delta_{Q''} }\right)$ then
$$\textbf{1}_{R}\leq \sum_{Q'' \in \mathbb{D}_m(Q')} \Phi_{Q''} \,\, \text{and} \,\, \sum_{Q'' \in \mathbb{D}_m(Q')} \Phi_{Q''} =1 \,\, \text{on} \,\, R,$$
and also,
$$\Phi_{Q''} \in \mathcal{C}^\infty_0(B_{Q''})\,\, \text{and} \,\, \|\nabla\Phi_{Q''}\|_{L^\infty} \lesssim \frac{1}{\ell(Q'')}.$$
If we define $$\eta_{Q'}:= \sum_{Q'' \in \widetilde{\mathbb{D}}_m(Q')}\Phi_{Q''},$$ then it is not hard to check that it satisfies the required properties.

\subsection{Proof of Lemma \ref{c3.lemma} for $p=2$}

Let us define
\begin{equation}\label{K(epsilon).p=2}
K(\varepsilon):= \sup_{Q}\frac{1}{\sigma(Q)}\int_{Q} \iint_{\Gamma_{Q,\varepsilon}(x)}|\Theta1(Y)|^2\,\frac{dY}{\delta(Y)^{n+1}}\,d\sigma(x)
\end{equation}
where
\begin{equation*}
\Gamma_{Q,\varepsilon}(x):=\bigcup_{\substack{Q'\ni x \\ Q'\subseteq Q \\ \varepsilon<\ell(Q')<\frac{1}{\varepsilon}}} \mathcal{U}_{Q'}\,\,\,\,\,\text{and}\,\,\,\,\, \gamma_{Q,\varepsilon}(x):=\bigcup_{\substack{Q'\ni x \\ Q'\in Good(Q) \\ \varepsilon<\ell(Q')<\frac{1}{\varepsilon}}}\mathcal{U}_{Q'}.
\end{equation*}

We fix a cube $Q$ and by splitting the domain of integration in order to exploit the size estimates of ${Q_k}$ and the square function bounds over the complement of a discretized ``sawtooth" region related to the family $\{Q_k\}$ we have that
\begin{align*}
\int_Q\iint_{\Gamma_{Q,\varepsilon}(x)}|\Theta1(Y)|^2\frac{dY}{\delta(Y)^{n+1}}d\sigma(x)
&\leq\sum_k\int_{Q_k}\iint_{\Gamma_{Q,\varepsilon}(x)}|\Theta1(Y)|^2\frac{dY}{\delta(Y)^{n+1}}d\sigma(x)\\
&+\int_{Q\setminus\bigcup Q_k}\iint_{\Gamma_{Q,\varepsilon}(x)}|\Theta1(Y)|^2\frac{dY}{\delta(Y)^{n+1}}d\sigma(x)\\
&=:\sum_k \I_k + \II.
\end{align*}

One may notice that for every $x\in Q\setminus\bigcup_kQ_k$, $\Gamma_Q(x)=\gamma_Q(x)$ and $\Gamma_{Q,\varepsilon}(x)=\gamma_{Q,\varepsilon}(x)$.

Therefore,
\begin{align}
|\II| &= \int_{Q\setminus\bigcup Q_k}\iint_{\gamma_{Q,\varepsilon}(x)}|\Theta1(Y)|^2\frac{dY}{\delta(Y)^{n+1}}d\sigma(x) \notag\\
&\leq\int_Q\int_{\gamma_Q(x)}|\Theta1(Y)|^2\frac{dY}{\delta(Y)^{n+1}}d\sigma(x) \notag \\
&\lesssim \sigma(Q),\label{c3.proof.of.lemma.II}
\end{align}
where in the last inequality we used \eqref{c3.theorem.eq2}. For the bound of $\I_k$ we fix a cube $Q_k$ and then observe that if $x\in Q_k$ then $\Gamma_{Q,\varepsilon}(x) \subseteq \Gamma_{Q_k,\varepsilon}(x) \cup \gamma_{Q,\varepsilon}(x)$, which in turn lets us have
\begin{align*}
\I_k &\leq \int_{Q_k}\displaystyle\iint_{\Gamma_{{Q_k},\varepsilon}}|\Theta1(Y)|^2\frac{dY}{\delta(Y)^{n+1}}d\sigma(x)+\int_{Q_k}\iint_{\gamma_{Q,\varepsilon}(x)}|\Theta1(Y)|^2\frac{dY}{\delta(Y)^{n+1}}d\sigma(x)\\
&=:\I_k^1+\I_k^2.
\end{align*}
Summing over all $k$'s, in light of the definition of $K(\varepsilon)$ and \eqref{c3.theorem.eq1}, we have that
\begin{equation}\label{c3.proof.of.lemma.I_k1}
\sum_k \I_k^1  \leq \sum_k \sigma(Q_k) K(\varepsilon) \leq (1-\eta)K(\varepsilon) \sigma(Q).
\end{equation}

Moreover, by our hypothesis \eqref{c3.theorem.eq2} and the fact that $\gamma_{Q,\varepsilon}(x)\subset\gamma_Q(x)$,
\begin{align} \sum_k \I_k^2 &= \int_{\cup_k  Q_k}\iint_{\gamma_{Q,\varepsilon}(x)}|\Theta1(Y)|^2\frac{dY}{\delta(Y)^{n+1}}d\sigma(x) \notag\\
&\leq \int_Q \iint_{\gamma_{Q}(x)}|\Theta1(Y)|^2\frac{dY}{\delta(Y)^{n+1}}d\sigma(x) \notag\\
&\lesssim \sigma(Q).\label{c3.proof.of.lemma.I_k2}
\end{align}

Therefore, by \eqref{c3.proof.of.lemma.II}, \eqref{c3.proof.of.lemma.I_k1} and \eqref{c3.proof.of.lemma.I_k2} we get
\begin{equation*}
\int_Q\iint_{\Gamma_{Q,\varepsilon}(x)}|\Theta1(Y)|^2\frac{dY}{\delta(Y)^{n+1}}d\sigma(x) \leq C \sigma(Q)+(1-\eta)K(\varepsilon) \sigma(Q),
\end{equation*}
for some constant $C>0$ and after we divide both sides by $\sigma(Q)$ and take the supremum over $Q$ we finally have that 
$$K(\varepsilon) \leq C + (1-\eta) K(\varepsilon).$$
By truncation, for each fixed $\varepsilon>0$, $K(\varepsilon)<\infty$, which in turn shows that,
$$ K(\varepsilon) \leq C/\eta,$$
uniformly in $\varepsilon$, by observing that the constants do not depend on $\varepsilon$. To this end, by letting $\varepsilon\to 0$ we obtain
$$\sup_Q\frac{1}{\sigma(Q)}\int_Q\iint_{\Gamma_{Q}(x)}|\Theta1(Y)|^2\frac{dY}{\delta(Y)^{n+1}}d\sigma(x)\leq C,$$
and conclude Lemma \ref{c3.lemma} for the special case $p=2$.

We now turn to the proof of Lemma \ref{c3.lemma} in full generality.

\subsection{Proof of Lemma \ref{3.1} for any $p \in (1,\infty)$} By H\"{o}lder's inequality and an argument of Fefferman and Stein \cite[p.146-147]{FS2} (since $b_Q$ need not be compactly supported) the case $p>2$ can be reduced to the case $p=2$ which was handled before. Therefore, we only need to prove Theorem \ref{main.theorem}, for $p$ ranging in the open interval $(1,2)$. To do so, we state and prove Lemma \ref{c4.sublemma} and then show that the conditions of Lemma \ref{c3.lemma} imply the conditions of Lemma \ref{c4.sublemma}. Let us fix $p\in(1,2)$ for the remainder of this section.

\begin{lemma}\label{c4.sublemma}

If there exist a positive and finite constant $N$ and a constant $0<\beta<1$, so that for every dyadic cube $Q$,
\begin{equation}\label{c4.sublemma.hypoth.}
\sigma\left(\left\{x\in Q: g_Q(x)>N \right\}\right)\leq (1-\beta)\,\sigma(Q),
\end{equation}
where 
\begin{equation*}
g_Q(x):=\left(\iint_{\Gamma_Q(x)}|\Theta 1(Y)|^2\frac{dY}{\delta(Y)^{n+1}}\right)^{p/2},
\end{equation*}
then \eqref{T1.Carleson.measure.cond.} holds.

\end{lemma}

\begin{proof}
Fix a dyadic cube Q and let us define 
$$g_{Q,\varepsilon}(x):=\left(\iint_{\Gamma_{Q,\varepsilon}(x)}|\Theta1(Y)|^2\frac{dY}{\delta(Y)^{n+1}}\right)^{p/2},$$
where $\Gamma_{Q,\varepsilon}(x)$ is as in \eqref{K(epsilon).p=2} and  set $D_{N,\varepsilon}:=\left\{ x\in Q: g_{Q,\varepsilon}(x)>N\right\}$. Since $\sigma$ is an outer regular Borel measure, we may choose $\mathcal{O}_{N,\epsilon}$, relatively open in Q, such that 
$$ D_{N,\epsilon} \subseteq \mathcal{O}_{N,\epsilon} \quad \text{and} \quad \sigma(\mathcal{O}_{N,\epsilon})<(1-\frac{\beta}{2})\sigma(Q).$$ 
We shall prove that if
\begin{equation*}
K(\varepsilon):=\sup_Q\frac{1}{|Q|}\int_Q g_{Q,\varepsilon}^{2/p}(x)d\sigma(x),
\end{equation*} 
then there exists a $C>0$ such that $K(\varepsilon)<C$, uniformly in $\varepsilon$. 

For, since $\mathcal{O}_{N,\varepsilon}$ is relatively open in Q we can decompose it via a stopping time argument to find a disjoint family of dyadic cubes $\{Q_j\}_j$ so that $\mathcal{O}_{N,\varepsilon} =\displaystyle\cup_j Q_j$, maximal with respect to inclusion in $\mathcal{O}_{N,\epsilon}$. If we set $Q\setminus\mathcal{O}_{N,\varepsilon}:=F_{N,\varepsilon}$, we have that
\begin{align*}
\int_Q g_{Q,\varepsilon}^{2/p}(x)d\sigma(x) &=\int_{F_{N,\varepsilon}}g_{Q,\varepsilon}^{2/p}(x)d\sigma(x)+\int_{\mathcal{O}_{N,\varepsilon}}g_{Q,\varepsilon}^{2/p}(x)d\sigma(x)\\
&\leq N^{2/p}\sigma(Q)+\int_{\mathcal{O}_{N,\varepsilon}}g_{Q,\varepsilon}^{2/p}(x)d\sigma(x).
\end{align*}
It is not hard to see that
\begin{align}
\int_{\mathcal{O}_{N,\varepsilon}}g_{Q,\varepsilon}^{2/p}(x)d\sigma(x) & \leq\displaystyle\sum_j\int_{Q_j}g_{Q_j,\varepsilon}^{2/p}(x)d\sigma(x)+\displaystyle\sum_j\int_{Q_j}\iint_{\gamma_{Q,\varepsilon}(x)}|\Theta1(Y)|^2\frac{dY}{\delta(Y)^{n+1}}d \sigma(x) \notag\\&=:I+II \label{c4.gQ.split}.
\end{align}
In view of \eqref{c4.sublemma.hypoth.}, since $Q_j$'s are mutually disjoint, 
\begin{align} 
I &=\sum_j \sigma(Q_j)\frac{1}{\sigma(Q_j)}\int_{Q_j}g_{Q_j,\varepsilon}^{2/p}(x)dx \notag\\
&\leq\sum_j \sigma(Q_j)K(\varepsilon) =  \sigma(\mathcal{O}_{N,\varepsilon})K(\varepsilon) \notag\\
&\leq \left(1-\frac{\beta}{2}\right) \sigma(Q) K(\varepsilon) \label{c4.sublemma.I}.
\end{align}
For the bound of $II$ it suffices to make the following observation. If we fix a cube $Q_j$ and $x\in Q_j$ then by the definition of our cones and the maximality of $Q_j$'s, there exists a point $x_j$ in the dyadic father of $Q_j$, say $Q_j^*$, such that $x_j\in F_{N,\varepsilon}$. Therefore, since $\gamma_{Q,\varepsilon}(x) \subseteq \gamma_{Q,\varepsilon}(x_j)$, we have that $g_{Q,\varepsilon}(x) \leq g_{Q,\varepsilon}(x_j)$, which in turn, entails

\begin{align} II  \leq \sum_j\int_{Q_j}\iint_{\gamma_{Q,\varepsilon}(x_j)}|\Theta1(Y)|^2\frac{dY}{\delta(Y)^{n+1}}d\sigma(x) &=\sum_j \sigma(Q_j) \left(g_{Q,\varepsilon}(x_j)\right)^{2/p} \notag\\ 
&\leq N^{2/p} \left(1-\frac{\beta}{2}\right)\sigma(Q).\label{c4.sublemma.II}\end{align}

In light of \eqref{c4.gQ.split}, \eqref{c4.sublemma.I} and \eqref{c4.sublemma.II},
\begin{equation*}\frac{1}{\sigma(Q)} \int_Q g_{Q,\varepsilon}^{2/p}(x)d\sigma(x) \leq (2-\beta) N^{2/p}+(1-\beta)K(\varepsilon), \end{equation*}
which after taking the supremum over all cubes $Q$, since $K(\varepsilon)<\infty$, we obtain
$$K(\varepsilon) \leq \frac{2-\beta}{\beta} N^{2/p}.$$
If we let $\varepsilon \searrow 0$, we conclude \eqref{T1.Carleson.measure.cond.}.
\end{proof}

We shall show now that the conditions of Lemma \ref{c3.lemma} imply the conditions of Lemma \ref{c4.sublemma}. Indeed. let us denote $D_N:=\{x\in Q:g_Q(x)>N\}$ for a large $N>0$ to be chosen. It is enough to prove that there exists $0<\beta<1$ such that
\begin{equation}\label{c4.eq.Cond.Lemma=>Cond.Sub}
\sigma(D_N) \leq (1-\beta)\sigma(Q).
\end{equation}
To this end,
\begin{align*}\sigma(D_N)  & \leq \sum_k \sigma(Q_k) +\sigma\left(\{x\in Q\setminus \bigcup Q_k:g_Q(x)>N\} \right)\\
&\leq (1-\eta) \sigma(Q) + \frac{C}{N} \int_{Q\setminus\bigcup Q_k} g_Q(x)d\sigma(x) \\
&\leq (1-\eta) \sigma(Q) + \frac{C}{N} \int_{Q\setminus\bigcup Q_k} \left(\iint_{\Gamma_Q(x)}|\Theta1(Y)|^2\frac{dY}{\delta(Y)^{n+1}}\right)^{p/2}d\sigma(x) \\
& =(1-\eta)\sigma(Q)+\frac{C}{N}\int_{Q\setminus\bigcup Q_k}\left(\iint_{\gamma_Q(x)}|\Theta1(Y)|^2\frac{dY}{\delta(Y)^{n+1}}\right)^{p/2}d\sigma(x) \\
& \leq (1-\eta)\sigma(Q)+\frac{C}{N}\int_{Q}\left(\iint_{\gamma_Q(x)}|\Theta1(Y)|^2\frac{dY}{\delta(Y)^{n+1}}\right)^{p/2}d\sigma(x) \\
& \leq (1-\eta)\sigma(Q) +\frac{C}{N}\sigma(Q),\end{align*}
where we have used \eqref{c3.theorem.eq2} and the fact that for any $x\in Q\setminus\bigcup Q_k$, $\Gamma_Q(x)=\gamma_Q(x)$. Now, if we have choose $N>0$ large enough so that $C/N\leq \eta/2$ and denote $\beta=\eta/2$,  \eqref{c4.eq.Cond.Lemma=>Cond.Sub} is concluded and so is Lemma \ref{c3.lemma}.

\subsection{Conclusion of Theorem \ref{main.theorem}}
After proving Lemma \ref{c3.lemma}, we are ready to deduce the conclusion of Theorem \ref{main.theorem} once we show a $T1$ theorem for square functions, i.e., it is enough to prove the following extension of a result of Christ and Journ\'e \cite{CJ} in the Euclidean case (which is exactly the case that $E=\mathbb{R}^n$ and $\Omega=\mathbb{R}_{+}^{n+1}$ or $\Omega=\mathbb{R}_-^{n+1})$.

\begin{theorem}[$T1$ Theorem]\label{T1theorem}
Let $\Theta f$ be defined as in \eqref{Theta} where $\psi(X,y)$ satisfies the conditions \eqref{psi.ptwise.estimate} and \eqref{psi.holder.cont}. If
\begin{equation}\label{T1.Carleson.measure.cond.}
\sup_Q \frac{1}{\sigma(Q)}\int_Q \iint_{\Gamma_Q(x)}\left|\Theta 1(Y) \right|^2 \frac{dY}{\delta(Y)^{n+1}}d\sigma(x) \lesssim 1,
\end{equation}
then
\begin{equation}
\iint_{\Omega}|\Theta f(Y)|^2\frac{dY}{\delta(Y)}\lesssim \int_E |f(x)|^2 d\sigma(x).
\end{equation}
\end{theorem}

\begin{proof}
It is not hard to see that if we decompose $\Omega$ using the ``Whitney boxes" $\mathcal{U}_Q$ associated with the ``dyadic grid" $\mathbb{D}$ and enlarge the area of integration, we have that
\begin{align*}
\iint_{\Omega}|\Theta f(Y)|^2\frac{dY}{\delta(Y)}&\lesssim \int_E\sum_k \sum_{\substack{Q\ni x \\ Q \in \mathbb{D}_k}} \iint_{\mathcal{U}_Q}|\Theta f(Y)|^2\frac{dY}{\delta(Y)^{n+1}}d\sigma(x).
\end{align*}

We shall use the Coifman-Meyer method and write
$$\Theta f= \left(\Theta f -\Theta1 \mathbb{E}_{Q}f\right)+ \Theta1  \mathbb{E}_{Q}f=: Rf+\Theta1 \mathbb{E}_{Q}f.$$

For the contribution of $Rf$ to the estimate, we introduce an average over the cube $Q$, $\frac{1}{\sigma(Q)}\int_Q d\sigma(x)$ and then apply Fubini and by the same method we used in the subsection \ref{3.1}, it can be easily seen that 
\begin{align*}\int_E\sum_k \sum_{\substack{Q\ni x \\ Q \in \mathbb{D}_k}}  \iint_{\mathcal{U}_Q}|Rf(Y)|^2\frac{dY}{\delta(Y)^{n+1}} d\sigma(x) &\leq \int_E \sum_k \sum_{Q \in \mathbb{D}_k} \textbf{1}_Q(x) \iint_{\mathcal{U}_Q}|R f(Y)|^2\frac{dY}{\delta(Y)^{n+1}}d\sigma(x)\\ &\lesssim \int_E |f(x)|^2 d\sigma(x).\end{align*}
Therefore, it is enough to prove that
\begin{equation}
\int_E\sum_k\sum_{Q \in \mathbb{D}_k}1_Q(x) \iint_{\mathcal{U}_Q}|\Theta1(Y)  \mathbb{E}_{Q}f|^2\frac{dY}{\delta(Y)^{n+1}}d\sigma(x) \lesssim \int_E \mathcal{M}(|f|)^2(x) d\sigma(x).
\label{T1.Embedding.bound}\end{equation}
For this, we shall make use of the discrete version of the Carleson's embedding lemma which is stated as follows.

\begin{lemma}(Discrete Carleson's embedding lemma)
If $\alpha_Q$ satisfy the Carleson condition 
\begin{equation} \sum_{Q\subseteq Q_0}\alpha_Q\sigma(Q)\lesssim \sigma(Q_0),\label{Carleson.measure}\end{equation}
for any cube $Q_0 \in \mathbb{D}$, then 
\begin{equation}\int_E\sum_{\substack{Q\ni x \\ Q \in \mathbb{D}}} \alpha_Q \left|\mathbb{E}_Q f \right|^2d\sigma(x) \lesssim \int_E \mathcal{M}(|f|)^2(x)d\sigma(x).\label{Carleson.conclusion}\end{equation}
\end{lemma}

To show \eqref{T1.Embedding.bound} let us set 
$$ \alpha_Q:=\iint_{\mathcal{U}_Q}|\Theta1(Y)|^2\frac{dY}{\delta(Y)^{n+1}}.$$
Therefore, by Fubini, 
\begin{align*}
\sum_{\substack{ Q \in \mathbb{D}\\Q \subseteq Q_0}}\alpha_Q \sigma(Q)
&= \int_E \textbf{1}_{Q_0}(x) \sum_{\substack{ Q \in \mathbb{D}\\{Q \subseteq Q_0}}} \textbf{1}_{Q}(x) \iint_{\mathcal{U}_Q}|\Theta1(Y)|^2\frac{dY}{\delta(Y)^{n+1}}\,d\sigma(x)\\
&= \int_{Q_0} \iint_{\Gamma_{Q_0}(x)}|\Theta1(Y)|^2\frac{dY}{\delta(Y)^{n+1}}\,d\sigma(x) \lesssim \sigma(Q_0),\end{align*}
where the last inequality follows from \eqref{T1.Carleson.measure.cond.}. So, $\alpha_Q$ satisfies \eqref{Carleson.measure} and by Carleson's embedding lemma, \eqref{T1.Embedding.bound} holds. This concludes the proof of Theorem \ref{T1theorem}.
\end{proof}

We now provide the reader with the proof of Carleson's lemma, although it is already known, in order to make our presentation more self-contained.
\begin{proof} (of the discrete Carleson's embedding lemma)
We may assume without loss of generality that $f\in L^2(E)$, given $\lambda>0$ there exists $r>0$ large enough, so that if $diam(Q)>r$ then 
$$\frac{1}{\sigma(Q)}\int_Q|f(x)|d\sigma(x)\leq \left(\frac{1}{\sigma(Q)}\int_Q |f(x)|^2d\sigma(x)\right)^{1/2} \leq \lambda.$$
If we perform a Calder\'on-Zygmund stopping time argument on each $Q$ we obtain a maximal family of cubes $\{Q_j\}_j$ with respect to the property $|\mathbb{E}_{Q_j}f|>\lambda$, and we have
$$\bigcup_j Q_j \subseteq \{x \in E : \mathcal{M}(|f|)(x)>\lambda\}.$$
To conclude the proof it is enough to show that if $\mathbb{A} \subseteq \mathbb{D}$ and $$\mu(\mathbb{A}):=\sum_{Q\in\mathbb{A}}\alpha_Q\sigma(Q),$$ then
\begin{equation}\label{Weak.Carleson}
\mu(\{Q: |\mathbb{E}_Qf|>\lambda\}) \leq C \|\mu\|_{\mathcal{C}}\, \sigma ( \{x\in E: \mathcal{M} (|f|) (x)>\lambda \} ),
\end{equation}
where $\| \cdot \|_\mathcal{C} $ stands for the Carleson norm of our measure which is defined by \eqref{Carleson.measure}. To this end, observe that 
$$\mathbb{A}_\lambda:= \{Q: |\mathbb{E}_Qf|>\lambda\} \subseteq \bigcup_j \mathcal{F}_j,$$
where $\mathcal{F}_j:= \{ Q \in \mathbb{A}_\lambda : Q \subseteq Q_j \}$. Therefore,
\begin{align*}
\mu(\mathbb{A}_\lambda)
&= \sum_j \sum_{\substack{Q\in\mathbb{A}_{\lambda} \\ Q \subseteq Q_j}} \alpha_Q\sigma(Q) \\
&\lesssim \sum_j \sigma(Q_j) = \sigma\big(\bigcup_j Q_j\big)\\
& \leq \sigma(\{x \in E : \mathcal{M}(|f|)(x)>\lambda\}).
\end{align*}
\end{proof}

\section{Theorem \ref{main.theorem} when $E$ is a bounded $ADR$ set} \label{sec.4}

In this section we prove Theorem \ref{main.theorem} in the case that the $ADR$ set $E$ is a bounded subset of $R^{n+1}$, which means that $diam(E)=:r_0<\infty$. Then, there exists an $n+1$-dimensional ball $B_{r_0}$ such that $E\subseteq B_{r_0}$. If we denote $B_{mr_0}$ to be the ball of radius $m r_0$ which is concentric with $B_{r_0}$, we have that

\begin{align*}
\iint_{\mathbb{R}^{n+1}\setminus E}\vert\Theta f(Y)\vert^2\frac{dY}{\delta(Y)} & = \iint_{B_{8r_0}\setminus E}\vert\Theta f(Y)\vert^2\frac{dY}{\delta(Y)}+\iint_{\mathbb{R}^{n+1}\setminus B_{8r_0}}\vert\Theta f(Y)\vert^2\frac{dY}{\delta(Y)}\\
& := I+II.
\end{align*}
Term $I$ may be handled by our previous arguments when $E$ is unbounded and the verification is left to the reader. Hence, to conclude the theorem, we only need to handle $II$. To this end,
\begin{align*}
II & \lesssim \sum_{k=3}^{\infty}\iint_{B_{2^{k+1}r_0}\setminus B_{2^k r_0}}\left(\int_E\frac{\delta(X)^{\alpha}}{|X-y|^{n+\alpha}}|f(y)|d\sigma(y)\right)^2\frac{dX}{\delta(X)}\\
&\approx\sum_{k=3}^{\infty}\frac{1}{(2^kr_0)^{2n}}\frac{1}{2^kr_0}(2^kr_0)^{n+1}\left(\int_E f(y)d\sigma(y)\right)^2\\
&\lesssim \sum_{k=3}^{\infty}2^{-kn}r_0^{-n}\sigma(E)\Vert f\Vert^2_{L^2(E)}\lesssim \Vert f\Vert_{L^2(E)}^2.
\end{align*}

\end{document}